\documentclass[paper=a4,english,fontsize=11pt,parskip=half,abstract=true]{scrartcl}
\usepackage{babel}
\usepackage[utf8]{inputenc}
\usepackage[T1]{fontenc}
\usepackage[left=20mm,right=20mm,top=30mm,bottom=30mm]{geometry}
\usepackage{amsmath}
\usepackage{amsthm}
\usepackage{amssymb}
\usepackage{enumerate} 
\usepackage{thmtools}
\usepackage{mathtools}
\mathtoolsset{centercolon} 
\usepackage[bookmarks=true,
            pdftitle={Fusion invariant characters of p-groups},
            pdfauthor={Benjamin Sambale},
            pdfkeywords={},
            pdfstartview={FitH}]{hyperref}

\newtheorem{Thm}{Theorem} 
\newtheorem{Lem}[Thm]{Lemma}

\theoremstyle{definition}

\numberwithin{equation}{section}

\allowdisplaybreaks[1]

\renewcommand{\phi}{\varphi}

\newcommand{\pcore}{\mathrm{O}}

\newcommand{\FC}{\mathcal{F}}

\newcommand{\Ind}{\mathrm{Ind}}
\newcommand{\ZZ}{\mathbb{Z}}

\newcommand{\QQ}{\mathbb{Q}}

\newcommand{\NN}{\mathbb{N}}
\newcommand{\FF}{\mathbb{F}}
\newcommand{\Aut}{\operatorname{Aut}}

\newcommand{\PGL}{\operatorname{PGL}}
\newcommand{\PSL}{\operatorname{PSL}}

\newcommand{\Irr}{\operatorname{Irr}}

\newcommand{\Ker}{\operatorname{Ker}}

\title{Fusion invariant characters of $p$-groups}
\author{Benjamin Sambale\footnote{Institut für Algebra, Zahlentheorie und Diskrete Mathematik, Leibniz Universität Hannover, Welfengarten 1, 30167 Hannover, Germany,
\href{mailto:sambale@math.uni-hannover.de}{sambale@math.uni-hannover.de}}}
\date{\today}

\begin{document}
\frenchspacing
\maketitle
\begin{abstract}\noindent
We consider complex characters of a $p$-group $P$, which are invariant under a fusion system $\FC$ on $P$. Extending a theorem of Bárcenas--Cantarero to non-saturated fusion systems, we show that the number of indecomposable $\FC$-invariant characters of $P$ is greater or equal than the number of $\FC$-conjugacy classes of $P$. We further prove that these two quantities coincide whenever $\FC$ is realized by a $p$-solvable group. On the other hand, we observe that this is false for constrained fusion systems in general. Finally, we construct a saturated fusion system with an indecomposable $\FC$-invariant character, which is not a summand of the regular character of $P$. This disproves a recent conjecture of Cantarero--Combariza.
\end{abstract}


\section{Introduction}

A fusion system $\mathcal{F}$ on a finite $p$-group $P$ is a category, whose objects are the subgroups of $P$ and whose morphisms are injective group homomorphisms satisfying certain technical conditions (we refer the reader to \cite{AKO} for the details). For the moment, we do not require that $\mathcal{F}$ is saturated. Elements $x,y\in P$ are called $\mathcal{F}$-\emph{conjugate} if there exists a morphism $f:\langle x\rangle \to P$ in $\FC$ such that $f(x)=y$. We denote the number of $\mathcal{F}$-conjugacy classes of $P$ by $k(\mathcal{F})$. 
A complex class function $\chi$ of $P$ is called $\mathcal{F}$-\emph{invariant} if $\chi$ is constant on the $\mathcal{F}$-conjugacy classes of $P$. These characters can often be used to construct new characters of finite groups via the Broué--Puig \mbox{$*$-construction} introduced in \cite{BrouePuigA}. Further motivation and background can be found in the recent paper of Cantarero--Combariza~\cite{Cantarero2}.

We call an $\mathcal{F}$-invariant character of $P$ \emph{indecomposable} if it is not the sum of two non-zero $\mathcal{F}$-invariant characters (this is unrelated to the characters of indecomposable modules of the group algebra). Let $\Ind_{\mathcal{F}}(P)$ be the set of indecomposable $\mathcal{F}$-invariant characters of $P$. 
In the theory of lattices, $\Ind_{\mathcal{F}}(P)$ is sometimes called the \emph{Hilbert basis} of the semigroup of $\FC$-invariant characters. As a consequence, $\Ind_{\mathcal{F}}(P)$ is finite (see \autoref{lemgordan} below).

Our first theorem gives a lower bound on $|\Ind_{\mathcal{F}}(P)|$. This was previously proved by Bárcenas and Cantarero in \cite[Lemma~2.1]{Cantarero} for saturated fusion systems. 

\begin{Thm}\label{thmHilbert}
The space of $\FC$-invariant class functions of $P$ is spanned by $\Ind_\FC(P)$. In particular, $|\Ind_{\mathcal{F}}(P)|\ge k(\mathcal{F})$. 
\end{Thm}

Cantarero and Combariza have proven in \cite[Lemma~2.17]{Cantarero2} that $|\Ind_\FC(P)|=k(\FC)$ holds for controlled fusion systems (among other cases). A controlled fusion system is realized by a group of the form $P\rtimes H$ for some $p'$-group $H$. Our second theorem generalizes this result to the larger class of $p$-solvable groups. 

\begin{Thm}\label{psolv}
Let $\FC$ be the (saturated) fusion system on a Sylow $p$-subgroup $P$ of a $p$-solvable group $G$. Then $|\Ind_\FC(P)|=k(\FC)$. 
\end{Thm}

In the last section of this paper, we construct examples of saturated constrained fusion systems with $|\Ind_\FC(P)|>|P|$ by making use of GAP~\cite{GAPnew}.
Since there are only finitely many fusion systems on a given $p$-group $P$, it is clear that $|\Ind_\FC(P)|$ can be bounded by a function in $|P|$. We do not know how to construct such a function explicitly. 
A related question for quasi-projective characters has been raised by Willems--Zalesski~\cite[Question~4.2]{WillemsZalesski}.

In \cite[Conjecture~2.19]{Cantarero2}, Cantarero and Combariza have conjectured that for a saturated fusion system, every $\mathcal{F}$-invariant indecomosable character is a summand of the regular character of $P$. 
In the last section, we exhibit a counterexample to this conjecture. 

\section{The number of indecomposable \texorpdfstring{$\FC$}{F}-invariant characters}

As in the introduction, $\FC$ denotes a fusion system on a finite $p$-group $P$ for the remainder of the paper. 
Our notation for characters follows Navarro's book~\cite{Navarro2}. In particular, if $\chi$ is a character of a group $G$ and $P\le G$, then $\chi_P$ denotes the restriction of $\chi$ to $P$. Moreover, for characters $\chi,\psi$ of $G$, the usual scalar product is denoted by $[\chi,\psi]$.

The following lemma is well-known among experts in lattice theory (it follows from \emph{Gordan's lemma}, see \cite[Theorem~16.4]{Schrijver}), but perhaps less known among representation theorists.

\begin{Lem}\label{lemgordan}
There are only finitely many indecomposable $\FC$-invariant characters of $P$.
\end{Lem}
\begin{proof}
Let $\Irr(P)=\{\chi_1,\ldots,\chi_k\}$. For $\psi\in\Ind_\FC(P)$ let $c(\psi)=([\psi,\chi_i]:i=1,\ldots,k)\in\NN_0^k$. We define a partial order on $\NN_0^k$ by $a\le b:\Longleftrightarrow b-a\in\NN_0^k$. For distinct characters $\psi,\psi'\in\Ind_\FC(P)$, we have $c(\psi)\nleq c(\psi')$, since otherwise $\psi'=(\psi'-\psi)+\psi$ would be a non-trivial decomposition of $\FC$-invariant characters. Therefore, $\{c(\psi):\psi\in\Ind_\FC(P)\}$ is an antichain in $\NN_0^k$ with respect to $\le$, i.\,e. no two distinct elements are comparable. Therefore, it is enough to show that every antichain in $\NN_0^k$ is finite. 

By way of contradiction, suppose that $c^{(1)},c^{(2)},\ldots$ is an infinite antichain in $\NN_0^k$. We may replace this sequence by an infinite subsequence such that $c^{(1)}_1\le c^{(2)}_1\le\ldots$. This sequence can in turn be replaced by a subsequence such that $c^{(1)}_2\le c^{(2)}_2\le\ldots$. Repeating this process $k$ times yields an infinite sequence $c^{(1)}\le c^{(2)}\le\ldots$. But this is impossible since the original sequence was an antichain. 
\end{proof}

Since for every $k\ge 2$, the poset $\NN_0^k$ contains antichains of arbitrary finite lengths (e.\,g. $(n,1,*,\ldots,*)$, $(n-1,2,*,\ldots,*),\ldots$ for any $n\in\NN$), it is not easy to give an upper bound on $|\Ind_\FC(P)|$.

We now prove the first theorem stated in the introduction.

\begin{proof}[Proof of \autoref{thmHilbert}]
By a theorem of Park~\cite{ParkExoticity2}, there exists a finite group $G$ such that $P\le G$ and the morphisms of $\FC$ are induced by conjugation in $G$. In particular, $k(\FC)$ is the number of $G$-conjugacy classes which intersect $P$. Let $T$ be the part of the character table of $G$, whose columns correspond to elements in $P$. Since the character table is invertible, $T$ has full rank. Hence, the ($G$-invariant) restrictions $\chi_P$ for $\chi\in\Irr(G)$ span the space of $G$-invariant class functions on $P$. Since each $\chi_P$ can be decomposed into $G$-invariant indecomposable characters, the claim follows.
\end{proof}

Next we restrict ourselves to saturated fusion systems arsing from a finite group with Sylow $p$-subgroup $P$ (those fusion systems are sometimes called \emph{non-exotic}). Here we can prove a stronger theorem, which resembles the fact that Brauer characters are restrictions of generalized characters (see \cite[Corollary~2.16]{Navarro}). 

\begin{Thm}\label{restriction}
Let $G$ be a finite group with Sylow $p$-subgroup $P$. Then every $G$-invariant character $\zeta$ of $P$ is the restriction of a generalized character of $G$. 
\end{Thm}
\begin{proof}
We extend $\zeta$ to a class function $\hat\zeta$ of $G$ in the following way: Every $g\in G$ is conjugate to an element of the form $xy=yx$ where $x\in P$ and $y$ is a $p'$-element. We define $\hat\zeta(g):=\zeta(x)$ (this is well-defined since $\zeta$ is $G$-invariant). Now we use Brauer's characterization of characters to show that $\hat\zeta$ is a generalized character of $G$ (see \cite[Corollary~7.12]{Navarro2}). To this end, let $N\le G$ be a nilpotent subgroup with Sylow $p$-subgroup $Q\unlhd N$. After conjugation, we may assume that $Q\le P$. Then $\hat\zeta_Q=\zeta_Q$ is a character of $Q\cong N/\pcore_{p'}(N)$ and $\hat\zeta_N$ is the inflation of $\zeta_Q$ to $N$. In particular, $\hat\zeta_N$ is a (generalized) character of $N$. Hence, $\hat\zeta$ is a generalized character of $G$, which restricts to $\zeta$. 
\end{proof}

Obviously, every $G$-invariant character of $P$ is a summand of a restriction of a character of $G$. However, an indecomposable character is not necessarily a summand of a restriction of an irreducible character of $G$. A counterexample will be given in the last section of the paper.

The following lemma of Cantarero--Combariza~\cite[Corollary~2.9]{Cantarero2} characterizes equality in \autoref{thmHilbert}.
We include the short proof for the convenience of the reader.

\begin{Lem}\label{lemunique}
For every fusion system $\FC$ on $P$ we have $|\Ind_\FC(P)|=k(\FC)$ if and only if every $\FC$-invariant character of $P$ can be decomposed \emph{uniquely} into indecomposable characters.
\end{Lem}
\begin{proof}
If $|\Ind_\FC(P)|=k(\FC)$, then $\Ind_\FC(P)$ is a basis of the space of $\FC$-invariant class functions and the result follows. 
Now assume that $|\Ind_\FC(P)|>k(\FC)$. Since the dimension of the $\QQ$-vectorspace spanned by $\Ind_\FC(P)$ is bounded by $k(\FC)$, the set $\Ind_\FC(P)$ is linearly dependent over $\QQ$. Hence, there exist integers $c_\psi\in\ZZ$ (not all zero) such that 
\[\sum_{\psi\in\Ind_\FC(P)}c_\psi\psi=0.\] 
Since the degree of each character is positive, not all $c_\psi$ can have the same sign. 
If we bring the negative coefficients to the right hand side, we end up with two distinct decompositions of an $\FC$-invariant character.
\end{proof}

We turn to the proof of our second main theorem.

\begin{proof}[Proof of \autoref{psolv}]
We apply Isaacs' theory of $\pi$-partial characters, where $\pi=\{p\}$ (see \cite[p.~71]{IsaacsSolv}). 
Every indecomposable $\FC$-invariant character $\chi$ of $P$ extends uniquely to a class function $\hat{\chi}$ on the set of $p$-elements of $G$. By \cite[Corollary~3.5]{IsaacsSolv}, $\hat{\chi}$ is an irreducible $p$-partial character of $G$. The number of those characters is exactly $k(\FC)$ by \cite[Theorem~3.3]{IsaacsSolv}.
\end{proof}

We remark that every fusion system of a $p$-solvable group is constrained. Conversely, by the model theorem~\cite[Theorem~I.4.9]{AKO}, every constrained fusion system is realized by a $p$-constrained group. However, \autoref{psolv} does not hold for constrained fusion systems in general as we will see in the next section.

As a consequence of \autoref{psolv}, we obtain the following extension of some results in \cite{Cantarero2}.

\begin{Thm}
Let $\FC$ be the (saturated) fusion system on a Sylow $p$-subgroup $P$ of a $p$-solvable group. Then every indecomposable $\FC$-invariant character of $P$ is a summand of the regular character of $P$. 
\end{Thm}
\begin{proof}
This follows from \autoref{psolv} and \cite[Remark~2.18]{Cantarero2}. For the convenience of the reader we repeat the short proof of the latter result: Let $\psi$ be an indecomposable $\FC$-invariant character of $P$. Let 
\[m:=\max\bigl\{[\psi,\chi]:\chi\in\Irr(P)\bigr\}.\] 
Then $\psi$ is a summand of $m\rho$, where $\rho$ is the regular character of $P$. By the hypothesis and \autoref{lemunique}, $m\rho$ has a unique decomposition into indecomposable $\FC$-invariant characters. Since $\rho$ itself is $\FC$-invariant (remember that $\rho(x)=0$ for all $x\in P\setminus\{1\}$), $\psi$ must appear as a summand of $\rho$. 
\end{proof}

\section{Counterexamples}

In \cite[table on p.~5206]{Cantarero2} and \cite{CantareroSym}, the authors list some fusion systems $\FC$ where $|\Ind_\FC(P)|>k(\FC)$, including the system on $P\cong D_{16}$ of the group $\PSL(2,17)$. This fusion system has two conjugacy classes of essential subgroups. The authors seem to have overlooked the “smaller” fusion system of $\PGL(2,7)$ with only one class of essential subgroups (still on $D_{16}$). With the notation 
\[P=\langle x,y\mid x^8=y^2=1,\ x^y=x^{-1}\rangle,\] 
the character table of $P$ is:
\[
\begin{array}{c|c|c|c|c|cc|c}
&1&x&x^3&x^2&x^4&y&xy\\\hline
\chi_1&1&1&1&1&1&1&1\\
\chi_2&1&-1&-1&1&1&1&-1\\
\chi_3&1&-1&-1&1&1&-1&1\\
\chi_4&1&1&1&1&1&-1&-1\\
\chi_5&2&0&0&-2&2&0&0\\
\chi_6&2&\sqrt{2}&-\sqrt{2}&0&-2&0&0\\
\chi_7&2&-\sqrt{2}&\sqrt{2}&0&-2&0&0
\end{array}
\]
We may assume that $x^4$ and $y$ are $\FC$-conjugate, but the other classes of $P$ are not fused. The $\FC$-invariant characters of $P$ must agree on the fifth and sixth column of the character table. Hence, we are looking for non-negative integral vectors orthogonal to $(0,0,1,1,1,-1,-1)$. Now it is easy to see that
\[\Ind_\FC(P)=\{\chi_1,\ \chi_2,\ \chi_3+\chi_6,\ \chi_3+\chi_7,\ \chi_4+\chi_6,\ \chi_4+\chi_7,\ \chi_5+\chi_6,\ \chi_5+\chi_7\}.\]
In particular, $|\Ind_\FC(P)|=8>6=k(\FC)$.

To turn this into a constrained fusion system, we set $G:=\PGL(2,7)$ and choose an irreducible faithful $\FF_2G$-module $V$ of dimension $6$. Then 
\[\hat{G}:=V\rtimes G=\mathrm{PrimitiveGroup}(64,64)=\mathrm{TransitiveGroup}(16,1802)\] 
(notation from GAP~\cite{GAPnew})
is a $2$-constrained group with Sylow $2$-subgroup $\hat{P}:=V\rtimes P$. Let $\hat\FC$ be the corresponding constrained fusion system. The inflations of the eight $G$-invariant indecomposable characters of $P$ are $\hat{\FC}$-indecomposable.
By the proof of \autoref{thmHilbert}, we may construct further $\hat\FC$-indecomposable characters by restricting characters $\chi\in\Irr(G)$ with $V\nsubseteq\Ker(\chi)$ to $\hat{P}$. The space spanned by those restrictions has dimension at least $k(\hat\FC)-k(\FC)=k(\hat\FC)-6$. In particular, $|\Ind_\FC(\hat{P})|>k(\hat\FC)$. 

Finally, we provide a counterexample to \cite[Conjecture~2.19]{Cantarero2} as claimed in the introduction. Let $\FC$ be the fusion system on a Sylow $2$-subgroup $P$ of the automorphism group of the Mathieu group $G=\Aut(M_{22})\cong M_{22}\rtimes C_2$. Then $|P|=2^8$. Let $\Irr(G)=\{\chi_1,\ldots\chi_{21}\}$ and $\Irr(P)=\{\lambda_1,\ldots,\lambda_{34}\}$. 
It can be checked with GAP that $k(\FC)=10$.
Let 
\[A:=\bigl([(\chi_j)_P,\lambda_i]\bigr)_{i,j}\in\ZZ^{34\times 21}.\]  
By \autoref{restriction}, every $\zeta\in\Ind_\FC(P)$ is the restriction of some generalized character $\psi$ of $G$. 
Setting $x:=([\psi,\chi_i])_i\in\ZZ^{21}$, we obtain $Ax=([\zeta,\tau_i])_i\ge 0$. Hence, $x$ belongs to the semigroup
\[S:=\{x\in\ZZ^{21}:Ax\ge0\}.\]
Moreover, since $\zeta$ is indecomposable, $x$ is a member of a Hilbert basis $H$ of $S$. We remark that $H$ is not unique, because there exist vectors $y$ with $Ay=0$. However, if $y\in H$ satisfies $Ax=Ay$, then $x=y$, since otherwise $x=(x-y)+y$ would be a non-trivial decomposition of $x$ in $S$. In this way, $H$ corresponds to $\Ind_\FC(P)$.
Using the \texttt{nconvex}-package~\cite{NConvex} in GAP, we compute $H$ and obtain $|H|=|\Ind_\FC(P)|=25$. The source code is available at \cite{github}. 
Moreover, $14$ indecomposable $\FC$-invariant characters are not summands of the regular character of $P$ and six are not summands of restrictions of irreducible characters of $G$. It would take too much space to print these characters here, but we exhibit at least one indecomposable character for illustration:
\[\zeta:=\lambda_1+\lambda_2+\lambda_3+2\lambda_4+\lambda_5+\lambda_6+\lambda_7+2\lambda_8+\lambda_9+2\lambda_{10}+\lambda_{11}+2\lambda_{12}.\]
The labeling is chosen in such a way that $\lambda_1,\ldots,\lambda_4$ have degree $1$, $\lambda_5,\ldots,\lambda_8$ have degree $2$, $\lambda_9,\lambda_{10}$ have degree $4$, and $\lambda_{11},\lambda_{12}$ have degree $8$. Since $\lambda_4$ occurs with multiplicity $2$, $\zeta$ is not a summand of the regular character of $P$.

The symmetric group $G=S_{12}$ is a counterexample for $p=2,3$. As promised in the introduction, $G=S_{10}$ for $p=2$ provides an example where $|\Ind_\FC(P)|=266>256=|P|$.

\section*{Acknowledgment}
I thank Gabriel Navarro for providing the idea for the proof of \autoref{psolv}. I became aware of the relevance of lattice theory (Hilbert bases) in representation theory by a talk of Wolfgang Willems in Jena, 2015. I appreciate a very detailed report (pointing out an error in a proof) by an anonymous referee.

\end{document}